\pdfoutput=1
\documentclass[11pt]{amsart}

\usepackage{epigamath}

\usepackage[english]{babel}

\usepackage{amssymb,amsmath,amsopn,amsxtra,amsthm,amsfonts,amscd}
\usepackage{xr}
\usepackage[mathcal]{euscript}
\usepackage{mathalfa}

\usepackage{tikz}
\usetikzlibrary{matrix}
\usepackage{tikz-cd}
\tikzset{shorten <>/.style={shorten >=#1,shorten <=#1}}

\usepackage{fixltx2e}

\usepackage{xspace}

\usepackage{xifthen}

\usepackage{xparse}

\newtheorem{cor}[subsection]{Corollary}
\newtheorem{lem}[subsection]{Lemma}
\newtheorem{prop}[subsection]{Proposition}

\newtheorem{conj}[subsection]{Conjecture}
\newtheorem{thm}[subsection]{Theorem}
\theoremstyle{definition}

\newtheorem{rem}[subsection]{Remark}

\newtheorem{ex}[subsection]{Example}

\newtheorem*{ackn}{Acknowledgments}

\theoremstyle{remark}

\newcommand{\secref}[1]{Section~\ref{#1}}

\renewcommand{\eqref}[1]{(\ref{#1})}

\renewcommand{\xspace}{}

\newcommand{\nc}{\newcommand}
\nc{\renc}{\renewcommand}
\nc{\ssec}{\subsection}
\nc{\sssec}{\subsubsection}
\nc{\on}{\operatorname}
\nc{\term}[1]{#1\xspace}

\DeclareMathSymbol{A}{\mathalpha}{operators}{`A}
\DeclareMathSymbol{B}{\mathalpha}{operators}{`B}
\DeclareMathSymbol{C}{\mathalpha}{operators}{`C}
\DeclareMathSymbol{D}{\mathalpha}{operators}{`D}
\DeclareMathSymbol{E}{\mathalpha}{operators}{`E}
\DeclareMathSymbol{F}{\mathalpha}{operators}{`F}
\DeclareMathSymbol{G}{\mathalpha}{operators}{`G}
\DeclareMathSymbol{H}{\mathalpha}{operators}{`H}
\DeclareMathSymbol{I}{\mathalpha}{operators}{`I}
\DeclareMathSymbol{J}{\mathalpha}{operators}{`J}
\DeclareMathSymbol{K}{\mathalpha}{operators}{`K}
\DeclareMathSymbol{L}{\mathalpha}{operators}{`L}
\DeclareMathSymbol{M}{\mathalpha}{operators}{`M}
\DeclareMathSymbol{N}{\mathalpha}{operators}{`N}
\DeclareMathSymbol{O}{\mathalpha}{operators}{`O}
\DeclareMathSymbol{P}{\mathalpha}{operators}{`P}
\DeclareMathSymbol{Q}{\mathalpha}{operators}{`Q}
\DeclareMathSymbol{R}{\mathalpha}{operators}{`R}
\DeclareMathSymbol{S}{\mathalpha}{operators}{`S}
\DeclareMathSymbol{T}{\mathalpha}{operators}{`T}
\DeclareMathSymbol{U}{\mathalpha}{operators}{`U}
\DeclareMathSymbol{V}{\mathalpha}{operators}{`V}
\DeclareMathSymbol{W}{\mathalpha}{operators}{`W}
\DeclareMathSymbol{X}{\mathalpha}{operators}{`X}
\DeclareMathSymbol{Y}{\mathalpha}{operators}{`Y}
\DeclareMathSymbol{Z}{\mathalpha}{operators}{`Z}

\nc{\sA}{\ensuremath{\mathcal{A}}\xspace}
\nc{\sB}{\ensuremath{\mathcal{B}}\xspace}
\nc{\sC}{\ensuremath{\mathcal{C}}\xspace}
\nc{\sD}{\ensuremath{\mathcal{D}}\xspace}
\nc{\sE}{\ensuremath{\mathcal{E}}\xspace}
\nc{\sF}{\ensuremath{\mathcal{F}}\xspace}
\nc{\sG}{\ensuremath{\mathcal{G}}\xspace}
\nc{\sH}{\ensuremath{\mathcal{H}}\xspace}
\nc{\sI}{\ensuremath{\mathcal{I}}\xspace}
\nc{\sJ}{\ensuremath{\mathcal{J}}\xspace}
\nc{\sK}{\ensuremath{\mathcal{K}}\xspace}
\nc{\sL}{\ensuremath{\mathcal{L}}\xspace}
\nc{\sM}{\ensuremath{\mathcal{M}}\xspace}
\nc{\sN}{\ensuremath{\mathcal{N}}\xspace}
\nc{\sO}{\ensuremath{\mathcal{O}}\xspace}
\nc{\sP}{\ensuremath{\mathcal{P}}\xspace}
\nc{\sQ}{\ensuremath{\mathcal{Q}}\xspace}
\nc{\sR}{\ensuremath{\mathcal{R}}\xspace}
\nc{\sS}{\ensuremath{\mathcal{S}}\xspace}
\nc{\sT}{\ensuremath{\mathcal{T}}\xspace}
\nc{\sU}{\ensuremath{\mathcal{U}}\xspace}
\nc{\sV}{\ensuremath{\mathcal{V}}\xspace}
\nc{\sW}{\ensuremath{\mathcal{W}}\xspace}
\nc{\sX}{\ensuremath{\mathcal{X}}\xspace}
\nc{\sY}{\ensuremath{\mathcal{Y}}\xspace}
\nc{\sZ}{\ensuremath{\mathcal{Z}}\xspace}

\nc{\bA}{\ensuremath{\mathbf{A}}\xspace}
\nc{\bB}{\ensuremath{\mathbf{B}}\xspace}
\nc{\bC}{\ensuremath{\mathbf{C}}\xspace}
\nc{\bD}{\ensuremath{\mathbf{D}}\xspace}
\nc{\bE}{\ensuremath{\mathbf{E}}\xspace}
\nc{\bF}{\ensuremath{\mathbf{F}}\xspace}
\nc{\bG}{\ensuremath{\mathbf{G}}\xspace}
\nc{\bH}{\ensuremath{\mathbf{H}}\xspace}
\nc{\bI}{\ensuremath{\mathbf{I}}\xspace}
\nc{\bJ}{\ensuremath{\mathbf{J}}\xspace}
\nc{\bK}{\ensuremath{\mathbf{K}}\xspace}
\nc{\bL}{\ensuremath{\mathbf{L}}\xspace}
\nc{\bM}{\ensuremath{\mathbf{M}}\xspace}
\nc{\bN}{\ensuremath{\mathbf{N}}\xspace}
\nc{\bO}{\ensuremath{\mathbf{O}}\xspace}
\nc{\bP}{\ensuremath{\mathbf{P}}\xspace}
\nc{\bQ}{\ensuremath{\mathbf{Q}}\xspace}
\nc{\bR}{\ensuremath{\mathbf{R}}\xspace}
\nc{\bS}{\ensuremath{\mathbf{S}}\xspace}
\nc{\bT}{\ensuremath{\mathbf{T}}\xspace}
\nc{\bU}{\ensuremath{\mathbf{U}}\xspace}
\nc{\bV}{\ensuremath{\mathbf{V}}\xspace}
\nc{\bW}{\ensuremath{\mathbf{W}}\xspace}
\nc{\bX}{\ensuremath{\mathbf{X}}\xspace}
\nc{\bY}{\ensuremath{\mathbf{Y}}\xspace}
\nc{\bZ}{\ensuremath{\mathbf{Z}}\xspace}

\nc{\dA}{\ensuremath{\mathds{A}}\xspace}
\nc{\dB}{\ensuremath{\mathds{B}}\xspace}
\nc{\dC}{\ensuremath{\mathds{C}}\xspace}
\nc{\dD}{\ensuremath{\mathds{D}}\xspace}
\nc{\dE}{\ensuremath{\mathds{E}}\xspace}
\nc{\dF}{\ensuremath{\mathds{F}}\xspace}
\nc{\dG}{\ensuremath{\mathds{G}}\xspace}
\nc{\dH}{\ensuremath{\mathds{H}}\xspace}
\nc{\dI}{\ensuremath{\mathds{I}}\xspace}
\nc{\dJ}{\ensuremath{\mathds{J}}\xspace}
\nc{\dK}{\ensuremath{\mathds{K}}\xspace}
\nc{\dL}{\ensuremath{\mathds{L}}\xspace}
\nc{\dM}{\ensuremath{\mathds{M}}\xspace}
\nc{\dN}{\ensuremath{\mathds{N}}\xspace}
\nc{\dO}{\ensuremath{\mathds{O}}\xspace}
\nc{\dP}{\ensuremath{\mathds{P}}\xspace}
\nc{\dQ}{\ensuremath{\mathds{Q}}\xspace}
\nc{\dR}{\ensuremath{\mathds{R}}\xspace}
\nc{\dS}{\ensuremath{\mathds{S}}\xspace}
\nc{\dT}{\ensuremath{\mathds{T}}\xspace}
\nc{\dU}{\ensuremath{\mathds{U}}\xspace}
\nc{\dV}{\ensuremath{\mathds{V}}\xspace}
\nc{\dW}{\ensuremath{\mathds{W}}\xspace}
\nc{\dX}{\ensuremath{\mathds{X}}\xspace}
\nc{\dY}{\ensuremath{\mathds{Y}}\xspace}
\nc{\dZ}{\ensuremath{\mathds{Z}}\xspace}

\nc{\bbA}{\ensuremath{\mathbb{A}}\xspace}
\nc{\bbB}{\ensuremath{\mathbb{B}}\xspace}
\nc{\bbC}{\ensuremath{\mathbb{C}}\xspace}
\nc{\bbD}{\ensuremath{\mathbb{D}}\xspace}
\nc{\bbE}{\ensuremath{\mathbb{E}}\xspace}
\nc{\bbF}{\ensuremath{\mathbb{F}}\xspace}
\nc{\bbG}{\ensuremath{\mathbb{G}}\xspace}
\nc{\bbH}{\ensuremath{\mathbb{H}}\xspace}
\nc{\bbI}{\ensuremath{\mathbb{I}}\xspace}
\nc{\bbJ}{\ensuremath{\mathbb{J}}\xspace}
\nc{\bbK}{\ensuremath{\mathbb{K}}\xspace}
\nc{\bbL}{\ensuremath{\mathbb{L}}\xspace}
\nc{\bbM}{\ensuremath{\mathbb{M}}\xspace}
\nc{\bbN}{\ensuremath{\mathbb{N}}\xspace}
\nc{\bbO}{\ensuremath{\mathbb{O}}\xspace}
\nc{\bbP}{\ensuremath{\mathbb{P}}\xspace}
\nc{\bbQ}{\ensuremath{\mathbb{Q}}\xspace}
\nc{\bbR}{\ensuremath{\mathbb{R}}\xspace}
\nc{\bbS}{\ensuremath{\mathbb{S}}\xspace}
\nc{\bbT}{\ensuremath{\mathbb{T}}\xspace}
\nc{\bbU}{\ensuremath{\mathbb{U}}\xspace}
\nc{\bbV}{\ensuremath{\mathbb{V}}\xspace}
\nc{\bbW}{\ensuremath{\mathbb{W}}\xspace}
\nc{\bbX}{\ensuremath{\mathbb{X}}\xspace}
\nc{\bbY}{\ensuremath{\mathbb{Y}}\xspace}
\nc{\bbZ}{\ensuremath{\mathbb{Z}}\xspace}

\nc{\mrm}[1]{\ensuremath{\mathrm{#1}}\xspace}
\nc{\mbf}[1]{\ensuremath{\mathbf{#1}}\xspace}
\nc{\mcal}[1]{\ensuremath{\mathcal{#1}}\xspace}
\nc{\msc}[1]{\ensuremath{\mathscr{#1}}\xspace}

\renc{\bar}[1]{\overline{#1}}

\let\sectsign\S
\let\S\relax

\nc{\sub}{\subset}
\nc{\too}{\longrightarrow}
\nc{\hook}{\hookrightarrow}
\nc*{\hooklongrightarrow}{\ensuremath{\lhook\joinrel\relbar\joinrel\rightarrow}}
\nc{\hooklong}{\hooklongrightarrow}
\nc{\twoheadlongrightarrow}{\relbar\joinrel\twoheadrightarrow}
\nc{\shiso}{\approx}
\nc{\isoto}{\xrightarrow{\sim}}
\nc{\isofrom}{\xleftarrow{\sim}}
\renc{\ge}{\geqslant}
\renc{\le}{\leqslant}
\renc{\geq}{\geqslant}
\renc{\leq}{\leqslant}

\nc{\id}{\mathrm{id}}

\DeclareMathOperator{\Hom}{\mathrm{Hom}}
\nc{\uHom}{\underline{\smash{\Hom}}}

\DeclareMathOperator{\Aut}{\mathrm{Aut}}
\DeclareMathOperator{\End}{\mathrm{End}}

\nc{\Pre}{\mathrm{PSh}{}}
\nc{\uEnd}{\underline{\smash{\End}}}

\renc{\lim}{\operatorname*{lim}}
\nc{\colim}{\operatorname*{colim}}
\nc{\Cofib}{\on{Cofib}}
\nc{\Fib}{\on{Fib}}
\nc{\initial}{\varnothing}
\nc{\op}{\mathrm{op}}

\let\bigcoprod=\coprod
\renc{\coprod}{\sqcup}

\nc{\bDelta}{\mbf{\Delta}}
\nc{\DM}{\mbf{DM}}
\nc{\eff}{\mathrm{eff}}
\nc{\veff}{\mathrm{veff}}
\nc{\cyc}{{\mrm{cyc}}}
\nc{\corr}{{\on{corr}}}
\nc{\fet}{{\mrm{f\acute et}}}
\nc{\fsyn}{{\mrm{fsyn}}}
\nc{\fqs}{{\mrm{fqs}}}
\nc{\syn}{{\mrm{syn}}}
\nc{\Perf}{\mbf{Perf}}
\nc{\Pic}{\mrm{Pic}}
\nc{\perf}{\mrm{perf}}
\nc{\oblv}{\on{oblv}}
\nc{\exact}{\on{exact}}

\nc{\F}{{\mbf{F}}}
\nc{\clopen}{{\mrm{clopen}}}
\nc{\B}{\mrm{B}}
\nc{\Fin}{\on{Fin}}
\nc{\Cut}{\on{Cut}}
\nc{\Cart}{\on{Cart}}
\nc{\pairs}{\mathsf{pairs}}
\nc{\Pairs}{\mathrm{Pair}}
\nc{\Trip}{\mathrm{Trip}}
\nc{\Lab}{\mathrm{Lab}}
\nc{\coCart}{\mathrm{coCart}}
\nc{\RKE}{\mathrm{RKE}}
\nc{\strict}{\mathrm{strict}}
\nc{\Emb}{\mathrm{Emb}}
\nc{\Split}{\mathrm{Split}}
\nc{\Set}{\mathrm{Set}}
\nc{\sSets}{\mathrm{sSets}}
\nc{\pb}{\mathrm{pb}}
\nc{\fib}{\mathrm{fib}}
\nc{\cofib}{\mathrm{cofib}}
\nc{\diff}{\mrm{diff}}
\nc{\gp}{\mrm{gp}}
\nc{\chr}{\mrm{char}}
\nc{\mgp}{\mrm{mot-gp}}
\nc{\FSyn}{\mrm{FSyn}}
\nc{\FEt}{\mrm{FEt}}
\nc{\Spc}{\mrm{Spc}}
\nc{\Ob}{\mrm{Ob}}
\nc{\Spt}{\mrm{Spt}}
\nc{\T}{\bT}
\nc{\suspinf}{\Sigma^\infty}
\nc{\h}{\mrm{h}}
\nc{\uhom}{\underline{\mathrm{Hom}}}
\nc{\umap}{\underline{\mathrm{Maps}}}
\renc{\H}{\bH}
\nc{\Einfty}{{\sE_\infty}}
\nc{\Eone}{{\sE_1}}
\nc{\Stab}{\mrm{Stab}}
\nc{\lax}{{\mrm{lax}}}
\nc{\cocart}{{\mrm{cocart}}}
\nc{\Sch}{\mrm{Sch}}
\nc{\dSch}{\mrm{dSch}}
\nc{\Aff}{\mrm{Aff}}
\nc{\SmAff}{\mrm{SmAff}}
\nc{\dAff}{\mrm{dAff}}
\nc{\Fr}{\on{Fr}}
\nc{\A}{\mathbf{A}}
\nc{\N}{\mathbf{N}}
\nc{\Z}{\mathbf{Z}}
\nc{\Q}{\mathbf{Q}}
\nc{\Oo}{\mathcal{O}} 
\nc{\red}{{\on{red}}}
\nc{\Voev}{{\on{Voev}}}
\nc{\Corr}{\mrm{Corr}}
\nc{\Span}{\mathbf{Corr}}
\nc{\Gap}{\mrm{Gap}}
\nc{\Filt}{\mrm{Filt}}
\nc{\Corrfr}{\Corr^{\fr}}
\nc{\Corrvfr}{\Corr^{\Vfr}}
\nc{\Spec}{\on{Spec}}
\nc{\Sm}{\mrm{Sm}}
\nc{\QSm}{\mrm{QSm}}
\nc{\Gm}{\mathbf{G}_{\mrm{m}}}
\renc{\P}{\bP}
\nc{\nis}{\mathrm{nis}}
\nc{\zar}{\mathrm{zar}}
\nc{\et}{\mathrm{\acute et}}
\nc{\all}{\mathrm{all}}
\nc{\fold}{\mathrm{fold}}
\nc{\Fun}{\mathrm{Fun}}
\nc{\Ho}{\mathrm{Ho}}
\nc{\Segal}{\mathrm{Segal}}
\nc{\Mon}{\mrm{Mon}{}}
\nc{\Ab}{\mrm{Ab}}
\nc{\Gr}{\mrm{Gr}}
\nc{\Sh}{\on{Sh}}
\nc{\M}{\mrm{M}}
\nc{\Lhtp}{L_{\A^1}}
\nc{\Lmot}{L_{\mrm{mot}}}
\nc{\mot}{\mrm{mot}}
\nc{\SH}{\mbf{SH}}
\nc{\RR}{\mbf{R}}
\nc{\CC}{\mbf{C}}
\nc{\Mod}{\mbf{Mod}}
\nc{\QCoh}{\mbf{QCoh}}
\nc{\MonUnit}{\mbf{1}}
\nc{\tr}{\on{tr}}
\nc{\vop}{\mrm{vop}}
\nc{\fr}{{\on{fr}}}
\nc{\Ar}{\mrm{Ar}}
\nc{\Vfr}{\on{Vfr}}
\nc{\frdiff}{{\on{frdiff}}}
\nc{\frGys}{\on{frGys}}
\nc{\SHfr}{\SH^{\fr}}
\nc{\SHfrdiff}{\SH^{\frdiff}}
\nc{\SHfrGys}{\SH^{\frGys}}
\nc{\InftyCat}{\infty\textnormal{-}\mrm{Cat}}
\nc{\TriCat}{\mathrm{TriCat}}
\nc{\Cat}{\mathrm{1\textnormal{-}Cat}}
\nc{\Th}{\on{Th}}

\nc{\CMon}{\mrm{CMon}{}}
\nc{\CAlg}{\mrm{CAlg}{}}
\nc{\KGL}{\mrm{KGL}}
\nc{\MGL}{\mrm{MGL}}
\nc{\MSL}{\mrm{MSL}}
\nc{\MSp}{\mrm{MSp}}
\nc{\Seg}{\mrm{Seg}{}}
\nc{\Tw}{\mrm{Tw}}
\nc{\sslash}{/\mkern-6mu/}
\nc{\PrL}{\mrm{Pr}^\mrm{L}}
\nc{\PrR}{\mrm{Pr}^\mrm{R}}
\nc{\pr}{\mrm{pr}}
\nc{\efr}{\mrm{efr}}
\nc{\nfr}{\mrm{nfr}}
\nc{\dfr}{\mrm{fr}}
\nc{\tfr}{\mrm{tfr}}
\nc{\Vect}{\mrm{Vect}}
\nc{\sVect}{\mrm{sVect}}
\nc{\fix}{\mrm{fix}}
\nc{\Hilb}{\mathrm{Hilb}}
\nc{\flci}{\mathrm{flci}}
\nc{\lci}{\mathrm{lci}}
\nc{\Isom}{\mathrm{Isom}}
\nc{\GL}{\mathrm{GL}}
\nc{\SL}{\mathrm{SL}}
\nc{\Sp}{\mathrm{Sp}}
\nc{\fin}{\mathrm{fin}}
\nc{\cl}{\mathrm{cl}}
\nc{\cn}{\mathrm{cn}}
\nc{\sm}{\mathrm{sm}}
\nc{\heart}{\heartsuit}
\nc{\1}{\mbf{1}}
\nc{\ornt}{\mrm{or}}
\nc{\GW}{\mrm{GW}}
\nc{\ev}{\mrm{ev}}
\nc{\FET}{\mathcal{FE}\mrm{t}}
\nc{\FSYN}{\mathcal{FS}\mrm{yn}}
\nc{\FQSM}{\mathcal{FQS}\mathrm{m}}
\nc{\fc}{\mrm{fc}{}}
\nc{\tel}{\mrm{tel}{}}
\nc{\lAngle}{\langle\!\langle}
\nc{\rAngle}{\rangle\!\rangle}

\let\phi\varphi

\EpigaVolumeYear{5}{2021} \EpigaArticleNr{9} \ReceivedOn{June 19,
2020}
\InFinalFormOn{January 25, 2021}
\AcceptedOn{February 26, 2021}

\title{On the infinite loop spaces of algebraic cobordism\\ and the motivic sphere}
\titlemark{On the infinite loop spaces of algebraic cobordism and the motivic sphere}

\author{Tom Bachmann}
\address{Mathematisches Institut,
LMU M\"unchen,
Theresienstr. 39,
80333 M\"unchen,
Germany} 
\email{tom.bachmann@zoho.com}

\author{Elden Elmanto}
\address{Department of Mathematics,
Harvard University,
1 Oxford St.,
Cambridge, MA 02138,
USA}
\email{elmanto@math.harvard.edu}

\author{Marc Hoyois}
\address{Fakult\"at f\"ur Mathematik,
Universit\"at Regensburg,
Universit\"atsstr. 31,
93040 Regensburg,
Germany}
\email{marc.hoyois@ur.de}

\author{Adeel A. Khan}
\address{IHES,
35 route de Chartres,
91440 Bures-sur-Yvette,
France;
Institute of Mathematics,
Academia Sinica,
Taipei 10617,
Taiwan}
\email{khan@ihes.fr}

\author{Vladimir Sosnilo}
\address{Laboratory ``Modern Algebra and Applications'',
St. Petersburg State University,
14th line, 29B,
199178 Saint Petersburg,
Russia}
\email{vsosnilo@gmail.com}

\author{Maria Yakerson}
\address{Institute for Mathematical Research (FIM),
ETH Z\"urich,
R\"amistr. 101,
8092 Z\"urich,
Switzerland}
\email{maria.yakerson@math.ethz.ch}

\authormark{T. Bachmann et al.}

\AbstractInEnglish{\footnotesize{We obtain geometric models for the infinite loop spaces of the motivic spectra $\MGL$, $\MSL$, and $\1$ over a field. They are motivically equivalent to $\Z\times \Hilb_\infty^\lci(\A^\infty)^+$, $\Z\times \Hilb_\infty^\ornt(\A^\infty)^+$, and $\Z\times \Hilb_\infty^\fr(\A^\infty)^+$, respectively, where $\Hilb_d^\lci(\A^n)$ (resp.\ $\Hilb_d^\ornt(\A^n)$, $\Hilb_d^\fr(\A^n)$) is the Hilbert scheme of lci points (resp.\ oriented points, framed points) of degree $d$ in $\A^n$, and $+$ is Quillen's plus construction. Moreover, we show that the plus construction is redundant in positive characteristic.}}



\TitleInFrench{Sur les espaces de lacets infinis du cobordisme alg\'ebrique et de la sph\`ere motivique}

\AbstractInFrench{\footnotesize{Nous obtenons des mod\`eles g\'eom\'etriques des espaces de lacets infinis
des spectres motiviques $\MGL$, $\MSL$ et $\1$ 
sur un corps. Ils sont motiviquement \'equivalents \`a $\Z\times \Hilb_\infty^\lci(\A^\infty)^+$, $\Z\times \Hilb_\infty^\ornt(\A^\infty)^+$ et $\Z\times \Hilb_\infty^\fr(\A^\infty)^+$, respectivement, o\`u $\Hilb_d^\lci(\A^n)$ (resp.\ $\Hilb_d^\ornt(\A^n)$, $\Hilb_d^\fr(\A^n)$) est le sch\'ema de Hilbert des points localement intersection compl\`ete (resp.\ points orient\'es, points cadr\'es) de degr\'e $d$ dans $\A^n$, et $+$ est la construction plus de Quillen. En outre, nous montrons que la construction plus est redondante en caract\'eristique positive.}}

\begin{document}

%%%%%%%%%%%%%%%%%%%%%%%%%%%%%%%
% Title page
%%%%%%%%%%%%%%%%%%%%%%%%%%%%%%%

\removeabove{1cm}
\removebetween{0.8cm}
\removebelow{0.8cm}

\maketitle

\begin{prelims}

\DisplayAbstractInEnglish

%\bigskip

\DisplayKeyWords

%\medskip

\DisplayMSCclass

%\bigskip

\languagesection{Fran\c{c}ais}

%\bigskip

\DisplayTitleInFrench

%\medskip

\DisplayAbstractInFrench

\end{prelims}

%%%%%%%%%%%%%%%%%%%%%
% Table of Contents
%%%%%%%%%%%%%%%%%%%%%

\newpage

\setcounter{tocdepth}{1}

\tableofcontents

%%%%%%%%%%%%%%%%%%%%%
% Content begins here
%%%%%%%%%%%%%%%%%%%%%

\section{Introduction}
Let $k$ be a field, let $\MGL$ be Voevodsky's algebraic cobordism spectrum over $k$, and let $\FSYN$ be the moduli stack of finite syntomic $k$-schemes. In \cite[Corollary 3.4.2]{EHKSY3}, we showed that there is an equivalence
\[
\Omega^\infty_\T\MGL\simeq L_\mot(\FSYN^\gp),
\]
where $\FSYN^\gp$ is the group completion of $\FSYN$ under disjoint union and $L_\mot$ is the motivic localization functor. 
This can be regarded as an algebro-geometric analogue of the identification of $\Omega^\infty\mathrm{MO}$ with the cobordism space of $0$-dimensional compact smooth manifolds. However, unlike in topology, the group completion is necessary in the above equivalence.
The goal of this paper is to obtain a group-completion-free description of $\Omega^\infty_\T\MGL$, which is more accessible for some purposes. 
Let us write $\FSYN_\infty$ for the colimit of the sequence
\[
\dotsb\to \FSYN_d \xrightarrow{+1} \FSYN_{d+1}\to\dotsb,
\]
where $\FSYN_d\subset \FSYN$ is the moduli stack of finite syntomic $k$-schemes of degree $d$.
Our main result is the following:

\begin{thm}\label{thm:intro}
	Let $k$ be a field. Then there is an equivalence
	\[
	\Omega^\infty_\T\MGL\simeq L_\mot(\Z\times \FSYN_\infty)^+
	\]
	in $\H(k)$, where $+$ denotes Quillen's plus construction (in the $\infty$-topos of Nisnevich sheaves). 
	Moreover, there is an equivalence
	\[
	(\Omega^\infty_\T\MGL)(k) \simeq \Z\times (\Lhtp\FSYN_\infty)(k)^+.
	\]
	If $k$ has positive characteristic, the same equivalences hold without the plus construction.
\end{thm}

Here, $\Lhtp$ is the ``naive'' $\A^1$-localization functor given by the formula
\[
(\Lhtp\sF)(X) = \colim_{n\in\Delta^\op}\sF(X\times\A^n).
\]

We recall that $\FSYN_d$ is motivically equivalent to the Hilbert scheme 
\[
\Hilb_d^\lci(\A^\infty) = \colim_{n\to\infty} \Hilb_d^\lci(\A^n)
\]
of finite local complete intersections of degree $d$ in $\A^\infty$, which is a smooth ind-scheme \cite[Lemma~3.5.1]{EHKSY3}. 
Taking the colimit along the obvious closed immersions 
\[
\Hilb_d^\lci(\A^\infty)\to \Hilb_{d+1}^\lci(\A^1\times \A^\infty) \simeq \Hilb_{d+1}^\lci(\A^\infty),
\]
we obtain a smooth ind-scheme $\Hilb_\infty^\lci(\A^\infty)$ such that the canonical map (forgetting the embedding into $\A^\infty$)
\[
\Hilb_\infty^\lci(\A^\infty) \to \FSYN_\infty
\]
is a motivic equivalence, and in fact an $\A^1$-equivalence on affine schemes. Thus, we can rewrite Theorem~\ref{thm:intro} as follows:

\begin{thm}
	\label{thm:intro-Hilb}
	Let $k$ be a field. Then there is an equivalence
	\[
	\Omega^\infty_\T\MGL\simeq L_\mot(\Z\times \Hilb_\infty^\lci(\A^\infty))^+
	\]
	in $\H(k)$. Moreover, there is an equivalence
	\[
	(\Omega^\infty_\T\MGL)(k) \simeq \Z\times (\Lhtp\Hilb_\infty^\lci(\A^\infty))(k)^+.
	\]
	If $k$ has positive characteristic, the same equivalences hold without the plus construction.
\end{thm}

\begin{rem}
	Unfortunately, we do not know if the plus construction can be removed in characteristic zero. As we will see, this is the case if and only if the cyclic permutation of three points in $\FSYN_3(k)$ becomes homotopic to the identity in $(\Lhtp\FSYN_\infty)(k)$.
\end{rem}

We have similar results for the motivic sphere spectrum $\1$ and for the special linear algebraic cobordism spectrum $\MSL$. Recall that a \emph{framing} of a finite syntomic morphism $f\colon Y\to X$ is a trivialization of the cotangent complex $\sL_f$ in $K(Y)$, and an \emph{orientation} of $f$ is a trivialization of the dualizing sheaf $\omega_f=\det(\sL_f)$ in $\Pic(Y)$. We denote by $\FSYN^\fr$ and $\FSYN^\ornt$ the moduli stacks of framed and oriented finite syntomic $k$-schemes, respectively. There are forgetful morphisms
\[
\FSYN^\fr \to \FSYN^\ornt\to \FSYN.
\]
By \cite[Theorem 3.5.18]{EHKSY1} and \cite[Corollary 3.4.4]{EHKSY3}, there are equivalences
\[
\Omega^\infty_\T\1 \simeq L_\mot(\FSYN^{\fr,\gp})\quad\text{and}\quad \Omega^\infty_\T\MSL \simeq L_\mot(\FSYN^{\ornt,\gp}).
\]
Let $h\in \FSYN^\fr(k)$ be the framed finite syntomic $k$-scheme consisting of two $k$-points, one with trivial framing and the other with the opposite framing. We define
\begin{gather*}
\FSYN^\fr_\infty=\colim(\dotsb\to \FSYN_{2d}^\fr \xrightarrow{+h} \FSYN_{2d+2}^\fr\to\dotsb),\\
\FSYN^\ornt_\infty=\colim(\dotsb\to \FSYN_{2d}^\ornt \xrightarrow{+h} \FSYN_{2d+2}^\ornt\to\dotsb).
\end{gather*}

\begin{thm}\label{thm:intro2}
	Let $k$ be a field. Then there are equivalences
	\begin{gather*}
	\Omega^\infty_\T\1\simeq L_\mot(\Z\times \FSYN_\infty^\fr)^+, \\
	\Omega^\infty_\T\MSL\simeq L_\mot(\Z\times \FSYN_\infty^\ornt)^+
	\end{gather*}
	in $\H(k)$. Moreover, there are equivalences
	\begin{gather*}
	(\Omega^\infty_\T\1)(k) \simeq \Z\times (\Lhtp\FSYN_\infty^\fr)(k)^+,\\
	(\Omega^\infty_\T\MSL)(k) \simeq \Z\times (\Lhtp\FSYN_\infty^\ornt)(k)^+.
	\end{gather*}
	If $k$ has positive characteristic, the same equivalences hold without the plus construction.
\end{thm}

Finally, we also have a Hilbert scheme version of Theorem~\ref{thm:intro2}: the moduli stacks $\FSYN^\fr_d$ and $\FSYN^\ornt_d$ are motivically equivalent to the smooth ind-schemes $\Hilb_d^\fr(\A^\infty)$ and $\Hilb_d^\ornt(\A^\infty)$ (see \cite[\sectsign 5.1.6]{EHKSY1} and \cite[\sectsign 3.5]{EHKSY3}). We define $\Hilb_\infty^\fr(\A^\infty)$ and $\Hilb_\infty^\ornt(\A^\infty)$ in an obvious way so as to match the definitions of $\FSYN^\fr_\infty$ and $\FSYN^\ornt_\infty$.

\begin{thm}
	\label{thm:intro2-Hilb}
	Let $k$ be a field. Then there are equivalences
	\begin{gather*}
	\Omega^\infty_\T\1\simeq L_\mot(\Z\times \Hilb_\infty^\fr(\A^\infty))^+, \\
	\Omega^\infty_\T\MSL\simeq L_\mot(\Z\times \Hilb_\infty^\ornt(\A^\infty))^+
	\end{gather*}
	in $\H(k)$. Moreover, there are equivalences
	\begin{gather*}
	(\Omega^\infty_\T\1)(k) \simeq \Z\times (\Lhtp\Hilb_\infty^\fr(\A^\infty))(k)^+,\\
	(\Omega^\infty_\T\MSL)(k) \simeq \Z\times (\Lhtp\Hilb_\infty^\ornt(\A^\infty))(k)^+.
	\end{gather*}
	If $k$ has positive characteristic, the same equivalences hold without the plus construction.
\end{thm}

Our interest in these group-completion-free models is motivated in part by the following conjecture of Mike Hopkins, called the ``Wilson space hypothesis'': 

\begin{conj}[M.\ Hopkins]
	\label{conj:hopkins}
	For every $n\in\Z$, the motive $M(\Omega^{\infty-n}_\T\MGL)\in\DM(k)$ is pure Tate.
\end{conj}

The analog of this conjecture for the complex cobordism spectrum $\mathrm{MU}$ is a theorem of Wilson \cite{Wilson}, and the analog in $C_2$-equivariant homotopy theory was recently proved by Hill and Hopkins \cite{HillHopkins}; these results are subsumed by Conjecture~\ref{conj:hopkins} for $k=\mathbf{C}$ and $k=\mathbf{R}$, respectively. Since the plus construction is invisible to motives, Theorem~\ref{thm:intro} implies that
\[
M(\Omega^\infty_\T\MGL) \simeq \bigoplus_{d \in \Z} M(\Hilb^\lci_\infty(\A^\infty)).
\]
We thus obtain the following geometric reformulation of the case $n=0$ of the above conjecture:

\begin{cor}
	Conjecture~\ref{conj:hopkins} holds for $n=0$ if and only if $M(\Hilb^\lci_\infty(\A^\infty))$ is a pure Tate motive.
\end{cor}

\begin{rem}
	The motive of $\Hilb^\lci(\A^n)$ is \emph{not} pure Tate for $2\leq n<\infty$.
	Indeed, one can show that $M(\Hilb_3^\lci(\A^n))$ cannot be pure Tate by analyzing the restriction to the lci locus of the Białynicki-Birula stratification of the Hilbert scheme $\Hilb_3(\A^n)$, which is smooth.
	For example,
	\[
	M(\Hilb_3^\lci(\A^2)) \simeq \Z \oplus \Z(1)[2] \oplus \Z(2)[4] \oplus \Z(4)[7].
	\]
	Nevertheless, the motive of $\Hilb_3^\lci(\A^\infty)$ turns out to be pure Tate, as shown in \cite[Theorem 7.1]{motive-hilb}.
\end{rem}

\begin{ackn}
We would like to thank Marc Levine for explaining to us the principle of moving via an étale correspondence; this was the main inspiration of Proposition \ref{prop:moving}.

The last five authors would like to acknowledge the hospitality of the Institute for Advanced Study in Princeton, where part of this work was done in July 2019.
\end{ackn}

\section{A moving lemma}

\begin{lem}\label{lem:curve-connectivity}
	Let $k$ be a field, $X$ a smooth geometrically connected quasi-projective $k$-scheme of dimension $\geq 2$, and $A\subset X$ a subscheme étale over $k$. Then there exists a smooth geometrically connected closed subscheme $H\subset X$ of codimension $1$ containing $A$.
\end{lem}

\begin{proof}
	If $k$ is infinite, a generic hypersurface section of $X$ of large enough degree containing $A$ is geometrically irreducible and smooth of codimension $1$, by \cite[Theorems 1 and 7]{AltmanKleiman}.
	
	If $k$ is finite, choose an embedding $X\subset \P^n_k$ and a first-order thickening $\tilde A$ of $A$ in $\P^n_k$ that is transverse to $X$. Let $P_d=\Gamma(\P^n_k,\sO(d))$ and let
	\begin{align*}
		Q_d &= \{f\in P_d\;|\; \text{$Z(f)\cap (X-A)$ is smooth of codimension $1$ in $X$ and $f|\tilde A = 0$}\},\\
		R_d &= \{f\in P_d\;|\; \text{$Z(f)\cap X$ is geometrically irreducible}\}.
	\end{align*}
	Note that $Z(f)\cap X$ is smooth for any $f\in Q_d$, so it suffices to show that $Q_d\cap R_d$ is nonempty for some $d$.
	We have
	\[
	\lim_{d\to\infty} \frac{\# Q_d}{\# P_d} > 0 \quad\text{and}\quad \lim_{d\to\infty}\frac{\# R_d}{\# P_d} = 1
	\]
	by \cite[Theorem 1.2]{Poonen} and \cite[Theorem 1.1]{CharlesPoonen}, respectively. Hence, $Q_d\cap R_d$ is nonempty for all large enough $d$.
\end{proof}

\begin{lem}\label{lem:bertini}
	Let $k$ be a field, $X\subset \P^n_k$ a quasi-projective $k$-scheme, and $A,B\subset X$ disjoint finite subschemes.
	Suppose $A$ étale over $k$ and $X-B$ smooth of dimension $m\geq 1$ over $k$. Then, for every large enough $d$, there exist global sections $f,g\in \Gamma(\P^n_k,\sO(d))$ such that:
	\begin{enumerate}
		\item[\rm (1)] $f|B=g|B$;
		\item[\rm (2)] $f|A=0$;
		\item[\rm (3)] $g$ does not vanish on $A\cup B$;
		\item[\rm (4)] $Z(f)\cap X$ and $Z(g)\cap X$ are smooth over $k$ of dimension $m-1$.
	\end{enumerate}
\end{lem}

\begin{proof}
	Suppose first that $k$ is finite. Let $h\in \Gamma(A\cup B,\sO(1))$ be a nonvanishing section. Applying \cite[Theorem 1.2]{Poonen} with an auxiliary thickening of $A$ as in the proof of Lemma~\ref{lem:curve-connectivity}, we deduce that for every large enough $d$ there exists a global section $f$ of $\sO(d)$ satisfying conditions (2) and (4) such that $f|B=h^{\otimes d}|B$. By the same theorem, there also exists a global section $g$ of $\sO(d)$ satisfying condition (4) such that $g|(A\cup B) = h^{\otimes d}$. Then $f$ and $g$ satisfy all conditions.
	
	If $k$ is infinite, take $d$ such that $\sI_A(d-1)$ is generated by its global sections and such that the vanishing of $H^1(\P^n_k,\sI_{A\cup B}(d))$ holds, where $\sI_Z\subset \sO_{\P^n_k}$ denotes the ideal defining a closed subscheme $Z$.
	Let $f$ and $h$ be generic sections of $\sO(d)$ with $f|A=0$ and $h|B=0$, and let $g=f+h$.
	By choice of $d$, every global section of $\sO(d)$ is the sum of one vanishing on $A$ and one vanishing on $B$. Hence, $g$ is a generic section of $\sO(d)$, and in particular does not vanish on $A\cup B$.
	Moreover, $Z(f)\cap (X-B)$ and $Z(g)\cap (X-B)$ are smooth of dimension $m-1$ by \cite[Theorem 7]{AltmanKleiman}. The sections $f$ and $g$ thus satisfy conditions (1)–(4).
\end{proof}

\begin{prop}\label{prop:moving}
	Let $k$ be a field, $X$ a smooth $k$-scheme, and $U\subset X$ a dense open subscheme. For any $k$-point $x\colon \Spec k\to X$, there exist framed correspondences $\alpha,\beta\in \Corr^\fr_k(\Spec k, U)$ with finite étale support and an $\A^1$-homotopy $x+\alpha\sim\beta$ in $\Corr^\fr_k(\A^1_k,X)$.
\end{prop}

\begin{proof}
	Since the question is local around $x$, we can assume $X$ quasi-projective and geometrically connected \cite[corollaire 4.5.14]{EGA4-2}.
 By Lemma~\ref{lem:curve-connectivity}, we can find a smooth connected subscheme $C\subset X$ of dimension $1$ such that $x\in C$ and such that $C\cap U$ is nonempty (hence cofinite in $C$).
 Let $\bar C$ be a projective closure of $C$ and $\sL$ an ample invertible sheaf on $\bar C$. Shrinking $C$ if necessary, we may assume that $C$ is affine and that $\Omega_{C/k}$ and $\sL|C$ are trivial.
 Let $A=\{x\}$, and let $B$ be the complement of $C\cap U$ in $\bar C$ with $x$ removed. Then $A$ and $B$ are disjoint finite subschemes of $\bar C$ such that $A$ is étale and $\bar C-B$ is smooth.
  By Lemma~\ref{lem:bertini}, we can find an integer $d$ and global sections $f$ and $g$ of $\sL^{\otimes d}$ with étale vanishing loci, such that $f(x)=0$, $f=g$ on $B$, and $Z(f)-\{x\},Z(g)\subset C\cap U$. Let $h$ be the section $(1-t)f+tg$ over $\A^1\times\bar C$ and let $H=Z(h)$. Then $H$ is a proper local complete intersection of relative virtual dimension $0$ over $\A^1_k$. Moreover, since $f$ and $g$ agree and do not vanish on $B$, $h$ does not vanish on $\A^1\times B$, so $H\subset \A^1\times C$ and $H$ is finite over $\A^1_k$ (hence syntomic by \cite[Proposition 2.1.16]{EHKSY1}). Thus, $H$ defines a finite syntomic correspondence from $\A^1_k$ to $X$. Since $\Omega_{C/k}$ and $\sL^{\otimes d}|C$ are trivial, there exists a framing $\tau\colon \sL_{H/\A^1_k}\simeq 0$ in $K(H)$. Multiplying $\tau$ by a unit of $k$, we can assume that $\tau$ restricts to the trivial loop in $K(\{x\})$. Then $(H,\tau)$ is a framed correspondence from $\A^1_k$ to $X$ with the desired properties.
\end{proof}

\section{Quillen's plus construction and group completion}
\label{sec:plus}

We shall say that a morphism $f\colon X\to Y$ in an $\infty$-topos $T$ is \emph{acyclic} if it is an epimorphism in the categorical sense, \textit{i.e.}, if the square
\[
\begin{tikzcd}
	X \ar{r}{f} \ar{d}[swap]{f} & Y \ar{d}{\id} \\
	Y \ar{r}[swap]{\id} & Y
\end{tikzcd}
\]
is cocartesian.
Recall that a \emph{modality} on $T$ is a factorization system on $T$ (in the sense of Joyal and as explained in \cite[\sectsign 5.2.8]{HTT}) whose left class of morphisms is stable under base change \cite[\sectsign 2]{ABFJ}.

\begin{lem}\label{lem:modality}
	In any $\infty$-topos $T$, acyclic morphisms form the left class of a modality.
	In particular, acyclic morphisms are stable under composition, retracts, colimits, cobase change, base change, and finite products.
\end{lem}

\begin{proof}
	Acyclic morphisms are stable under base change since cocartesian squares are (by universality of colimits).
	By \cite[Proposition 5.5.5.7]{HTT}, it remains to show that the class of acyclic maps is of small generation as a saturated class.
	For any $X\in T$, the full subcategory of $T_{/X}$ spanned by the acyclic morphisms is accessible by \cite[Proposition 5.4.6.6]{HTT}, being the fiber of the suspension functor. It is thus generated under filtered colimits by a small subcategory. Let $C\subset \Fun(\Delta^1,T)$  be the union of these small subcategories of $T_{/X}$ as $X$ ranges over a small set of generators of $T$. Using that acyclic maps are stable under base change, we immediately deduce that $C$ generates the class of acyclic maps under colimits.
\end{proof}

For $X\in T$, we denote by $X\to X^+$ the final object in the $\infty$-category of acyclic maps out of $X$, which exists by Lemma~\ref{lem:modality}. The functor $X\mapsto X^+$ is called the \emph{plus construction}.
Note that acyclic morphisms are connected and become equivalences after a single suspension; it follows that the canonical map $X\to X^+$ induces an isomorphism on $\pi_0$, and that it is an equivalence whenever $X$ admits a structure of $\sE_1$-group (\textit{i.e.}, when $X$ is a loop space \cite[Theorem 5.2.6.15]{HA}).

In the $\infty$-topos $\Spc$ of spaces, the above construction coincides with Quillen's plus construction which was used to define algebraic $K$-theory \cite[\sectsign 12]{quillen-icm}. More precisely, $X\to X^+$ is the initial map that kills the maximal perfect subgroups of the fundamental groups of $X$, and hence it is an equivalence if and only if the fundamental groups of $X$ are hypoabelian (\textit{i.e.}, have no nontrivial perfect subgroups). We refer to \cite{Raptis} for a discussion of acyclic morphisms in $\Spc$ and for a proof of this fact.

Let $M$ be a commutative monoid in $T$ and let $m\colon *\to M$ be a global section. We denote by $\Mod_M(T)$ the $\infty$-category of $M$-modules, \textit{i.e.}, objects of $T$ with an action of $M$ (which is again an $\infty$-topos). The full subcategory of $M$-modules on which $m$ acts invertibly is reflective, and we denote by $E\mapsto E[m^{-1}]$ the associated localization functor. It is easy to check that this functor preserves finite products. In particular, $M[m^{-1}]$ is again a commutative monoid and $M\to M[m^{-1}]$ is the initial morphism of commutative monoids sending $m$ to an invertible global section. For $E\in \Mod_M(T)$, we define the telescope $\tel_m(E) \in \Mod_M(T)$ as the colimit of the sequence
\[
\dotsb \to E \xrightarrow{\cdot m} E \to \dotsb.
\]
There is then a canonical map of $M$-modules $\tel_m(E)\to E[m^{-1}]$, but unlike in the case of $1$-topoi, it is not invertible in general. For example, if $T=\Spc$ and $F=\bigcoprod_{n\geq 0}B\Sigma_n$ is the free commutative monoid on a single element $a$, then $F[a^{-1}]$ is the group completion of $F$, but $\tel_a(F)\simeq \Z\times B\Sigma_\infty$ does not admit a monoid structure since its fundamental groups are not abelian.

\begin{prop}\label{prop:group completion}
	Let $T$ be an $\infty$-topos, let $M$ be a commutative monoid in $T$, and let $m\colon *\to M$ be a global section.
	For any $M$-module $E$, the canonical map $\tel_m(E) \to E[m^{-1}]$ is acyclic.
\end{prop}

\begin{proof}
	Since acyclic maps are closed under colimits and finite products, it suffices to prove the result for $E=M$.
	The classifying $\infty$-topos for pointed commutative monoids is a presheaf $\infty$-topos (namely, presheaves on the pushout of $(\Spc_*^\fin)^\op\leftarrow (\Spc^\fin)^\op \to (\CMon(\Spc)^\fin)^\op$ in the $\infty$-category of $\infty$-categories with finite limits), so we are immediately reduced to the case $T=\Spc$. Let $F$ be the free commutative monoid on an element $a$. The element $m$ induces a morphism of commutative monoids $F\to M$ sending $a$ to $m$, and we have $\tel_m(M)=M\otimes_F \tel_a(F)$ and $M[m^{-1}]\simeq M\otimes_F F^\gp$, where $\otimes_F$ is the tensor product of spaces with $F$-action. Since acyclic maps are closed under colimits and finite products, $M\otimes_F(-)$ preserves acyclic maps. We can therefore replace $(M,m)$ by $(F,a)$, and in particular we can assume that $\pi_0(M)[m^{-1}]$ is a group. In this case we must show that the canonical map $\tel_m(M)^+ \to M^\gp$ is an equivalence, which is the classical McDuff–Segal group completion theorem \cite{mcduff-segal}. We recall a proof due to Nikolaus \cite{Nikolaus}.
	Note that the plus construction preserves finite products and hence commutative monoids.
	 Consider the commutative square
\[
\begin{tikzcd}
	\tel_m(M)^+ \ar{r}\ar{d} & M^\gp \ar{d} \\
	\tel_m(M^+)^+ \ar{r} & (M^+)^\gp\rlap.
\end{tikzcd}
\]
The left vertical is an equivalence since the plus construction is a left localization of $\Spc$. The lower horizontal map is an equivalence by Proposition~\ref{prop:robalo}, since the cyclic permutation of order $5$ becomes trivial in the hypoabelianization of $\Sigma_5$. The top horizontal map is a stable equivalence by the localization theory of $\Einfty$-ring spectra. Hence the right vertical map is a stable equivalence. Since $\Einfty$-groups are simple, the right vertical map is in fact an equivalence, so the top horizontal map is an equivalence as well.
\end{proof}

\begin{rem}
	In the proof of Proposition~\ref{prop:group completion}, the reduction to the McDuff–Segal theorem only uses that $M$ is an $\sE_2$-monoid, and the given proof of that theorem only uses that $M$ is $\sE_3$. In fact, the McDuff–Segal theorem holds more generally for homotopy commutative $\sE_1$-monoids \cite{ORW}, so Proposition~\ref{prop:group completion} holds for $M$ an $\sE_2$-monoid.
\end{rem}

If $S$ is an arbitrary set of global sections of $M$, we can define more generally the telescope functor $\tel_S\colon \Mod_M(T)\to \Mod_M(T)$ by
\[
\tel_S(E) = \colim_{\substack{F\subset S\\F\text{ finite}}} \tel_{\prod_{m\in F}m}(E)
\]
(see \cite[\sectsign 6.1]{hoyois-sixops} for a precise construction). Proposition~\ref{prop:group completion} immediately implies that the canonical map $\tel_S(E)\to E[S^{-1}]$ is acyclic. In particular:

\begin{cor}\label{cor:group completion}
	Let $T$ be an $\infty$-topos, $M$ a commutative monoid in $T$, and $S$ a set of global sections of $M$ such that $\pi_0(M)[S^{-1}]$ is a group. Then there is an equivalence $\tel_S(M)^+\simeq M^\gp$.
\end{cor}

In what follows we will use the plus construction in $\Spc$ and in the $\infty$-topos of Nisnevich sheaves on $\Sm_k$.

\section{Proofs of the main results}

Recall that $\FSYN_\infty=\colim_{d\to\infty}\FSYN_d$.
In the notation of \secref{sec:plus}, there is a canonical map
\[
\Z\times\FSYN_\infty\to \tel_{1}(\FSYN),
\]
which is an equivalence on connected schemes (in particular, it is a Zariski-local equivalence).

\begin{prop}\label{prop:main}
	Let $k$ be a field. Then $L_\mot\FSYN_{\infty}\in \H(k)$ is connected in the Nisnevich topology and $(\Lhtp\FSYN_\infty)(k)$ is connected.
\end{prop}

\begin{proof}
	For the first statement, by \cite[Lemma 6.1.3]{morel-connectivity}, it suffices to show that the stalks of $L_\mot\FSYN_\infty$ at separable finitely generated field extensions of $k$ are connected. Since there is a surjection
	\[
	\pi_0(\Lhtp\FSYN_\infty)(k) \to \pi_0(L_\mot\FSYN_\infty)(k)
	\]
	by \cite[\sectsign2, Corollary 3.22]{MV}, we are reduced to proving the second statement. We will show more precisely that for any $T\in\FSYN_d(k)$ there exists $n\geq 0$ and a cobordism $T+n\sim d+n$ in $\FSYN_{d+n}$. 
	Choosing an embedding of $T$ in $\A^m$ defines a $k$-point of the smooth $k$-scheme $\Hilb_d^\lci(\A^m)$. 
	We now apply Proposition~\ref{prop:moving} with $X=\Hilb_d^\lci(\A^m)$ and $U$ the open subscheme classifying finite étale subschemes of $\A^m$ of degree $d$. We obtain in particular a finite syntomic correspondence $\A^1\leftarrow H\to \Hilb_d^\lci(\A^m)$ such that the fibers $H_0$ and $H_1$ are étale over $k$, $H_0=\{x\}\sqcup H_0'$, and $H_0'$ and $H_1$ map to $U$. The map $H\to \Hilb_d^\lci(\A^m)$ classifies a finite syntomic morphism $W\to H$, and the composite $W\to H\to \A^1_k$ is a cobordism between $W_0\simeq T\sqcup W_0'$ and $W_1$ such that $W_0'$ and $W_1$ are finite étale over $k$. Using \cite[Proposition B.1.4]{EHKSY1}, both $W_0'$ and $W_1$ are further cobordant to their degree, which concludes the proof.
\end{proof}

\begin{cor}\label{cor:main}
	Let $k$ be a field. Then the canonical maps
	\begin{gather*}
	L_\mot(\Z\times\FSYN_\infty)^+ \to L_\mot(\FSYN^\gp)\\
	\Z\times (\Lhtp\FSYN_\infty)(k)^+ \to (\Lhtp\FSYN)(k)^\gp
	\end{gather*}
	are equivalences.
\end{cor}

\begin{proof}
	We first observe that $L_\mot(\FSYN^\gp)$ is the group completion of $L_\mot(\FSYN)$ in Nisnevich sheaves. To prove this we may replace $k$ by a perfect subfield, and the claim follows from \cite[Theorem 3.4.11]{EHKSY1} since the objectwise group completion of $L_\mot(\FSYN)$ is an $\A^1$-invariant presheaf with framed transfers.
	By Corollary~\ref{cor:group completion}, it remains to show that the monoids $\pi_0^\nis L_\mot(\FSYN_\infty)$ and $\pi_0(\Lhtp\FSYN_\infty)(k)$ are groups, which follows from Proposition~\ref{prop:main}.
\end{proof}

Except for the statement about positive characteristic, which we shall prove in \secref{sec:remove-plus}, Theorem~\ref{thm:intro} follows from Corollary~\ref{cor:main} and \cite[Corollary 3.4.2(i)]{EHKSY3}.

If $S$ is a scheme and $a\in\sO(S)^\times$, we denote by $\langle a\rangle\in \FSYN^\fr(S)$ the finite syntomic $S$-scheme $S$ framed by the image of $a$ under the canonical map $\sO(S)^\times \to\Omega K(S)$. We write $n_\epsilon$ for the alternating sum
\[
\langle 1\rangle + \langle -1\rangle + \langle 1\rangle +\dotsb
\]
with $n$ terms, and we write $h$ for $2_\epsilon=\langle 1\rangle+\langle -1\rangle$. Recall that we have $\FSYN_\infty^\fr=\colim_{d\to\infty}\FSYN_{2d}^\fr$ and $\FSYN_\infty^\ornt=\colim_{d\to\infty}\FSYN_{2d}^\ornt$, where the transition maps are given by adding $h$. In the notation of \secref{sec:plus}, we have equivalences
\[
\Z\times\FSYN_\infty^\fr\simeq \tel_h(\FSYN^\fr)
\quad\text{and}\quad 
\Z\times\FSYN_\infty^\ornt\simeq \tel_h(\FSYN^\ornt)
\]
on connected schemes.

We recall a basic construction from \cite[\sectsign B.1.1]{EHKSY1}. Let $S$ be a scheme and $f(x)\in\sO(S)[x]$ a polynomial with invertible leading coefficient. Then the zero locus $Z(f)\subset \A^1_S$ is a finite syntomic $S$-scheme. Moreover, $f$ generates the conormal sheaf $\sN_{Z(f)/\A^1_S}=(f)/(f^2)$, hence defines a framing of $Z(f)$ over $S$. We denote this framed finite syntomic $S$-scheme by $\phi(f)$. By \cite[Proposition B.1.4]{EHKSY1}, $\phi(f)$ is framed cobordant to $\langle a\rangle d_\epsilon$, where $a$ is the leading coefficient of $f(x)$ and $d$ is its degree.

\begin{lem}\label{lem:GW-relation}
	Let $S$ be a scheme and $a\in\sO(S)^\times$. Then there exists an $\A^1$-homotopy in $\FSYN_S^\fr$ between $\langle a\rangle+\langle -a\rangle$ and $h=\langle 1\rangle+\langle -1\rangle$.
\end{lem}

\begin{proof}
	We have $\langle a\rangle+\langle -a\rangle\simeq \phi(x(x-a))$, because $x(x-a)$ is congruent to $a(x-a)$ modulo ${(x-a)^2}$ and to $-ax$ modulo ${x^2}$. Thus, if $h(t)=1+t(a-1)$, then $\phi(x(x-h(t)))$ is the desired $\A^1$-homotopy.
\end{proof}

\begin{rem}
	The main point of Lemma~\ref{lem:GW-relation} is that the relation $\langle a\rangle+\langle -a\rangle=\langle 1\rangle+\langle -1\rangle$ in the Grothendieck–Witt group holds in $\Lhtp\FSYN_S^\fr$ prior to group completion. This is not the case of some other relations, such as $\langle ab^2\rangle=\langle a\rangle$ (indeed, $\FSYN^\fr_1\simeq \Omega K$ is already $\A^1$-invariant on regular schemes).
\end{rem}

We denote by $\GW$ the presheaf of Grothendieck--Witt rings \cite{knebusch-bilinear}, and by $\underline{\GW} = \underline{K}^{MW}_0$ the presheaf of unramified Grothendieck--Witt rings \cite[\sectsign 3.2]{Morel}.
Note that if $K$ is a field then $\underline{GW}(K) = \GW(K)$ \cite[Lemma 3.10]{Morel}, whereas in general the relationship between $\underline{GW}(R)$ and $\GW(R)$ is more subtle.

\begin{lem}\label{lem:GW-generators}
	Let $R$ be a regular local ring over a field. Then the group $\underline{\GW}(R)$ is generated by $\langle a\rangle$ for $a\in R^\times$.
\end{lem}

\begin{proof}
	If $R$ has characteristic $\neq 2$, we know that $\underline{\GW}(R)=\GW(R)$ \cite[Theorem 10.12]{norms}, and $\GW(R)$ has the claimed property by \cite[Corollary I.3.4]{HM}.
	Suppose therefore that $R$ has characteristic $2$. By Popescu's theorem \cite[Tag 07GC]{stacks}, we can assume $R$ essentially smooth over a perfect field. The result now follows from Lemma \ref{lem:I-generators}(iii) below.
\end{proof}

\begin{lem}\label{lem:I-generators}
	Let $k$ be a perfect field of characteristic $2$. Let $I \subset \underline{\GW}$ denote the kernel of the rank map $\underline{\GW} \to \Z$, and write $I^n \subset \underline{\GW}$ for the $n$th power Zariski subsheaf of ideals.
	\begin{itemize}
	\item[\rm (i)]
	Each sheaf $I^n$ is strictly $\A^1$-invariant $($in particular a Nisnevich sheaf\,$)$.
	\item[\rm (ii)]
	We have $L_\zar(I^n/I^{n+1})\simeq \nu^n$, where $\nu^n$ is the sheaf of logarithmic differential $n$-forms \cite[\sectsign 2]{geisser-levine}.
	\item[\rm (iii)]
	For $n \geq 1$, the sheaf $I^n$ is Zariski locally generated by $n$-fold Pfister forms $\lAngle a_1, \dots, a_n \rAngle$ with $a_i \in \sO^\times$. The sheaf $\underline{\GW}$ is Zariski locally generated by $\langle a \rangle$ for $a \in \sO^\times$.
	\end{itemize}
\end{lem}

\begin{proof}
Denote by $\underline{I}^n \subset \underline{\GW}$ the sheaf defined in \cite[Example 3.34]{Morel}. Thus each $\underline{I}^n$ is a strictly $\A^1$-invariant subsheaf of ideals, $\underline{I}^n\underline{I}^m \subset \underline{I}^{n+m}$, and $\underline{I}^n(K) = I^n(K)$ for finitely generated field extensions $K/k$. Moreover $\underline{I}^1$ is the kernel of the rank map, so that $\underline{I}^1 = I$.
For $X\in\Sm_k$, write $N^n(X) \subset \nu^n(X)$ for the subgroup generated by global logarithmic differentials, \textit{i.e.}, expressions of the form $[a_1, \dotsc, a_n] = da_1/a_1 \wedge \dotsb \wedge da_n/a_n$ with $a_i \in \sO(X)^\times$. It is clear that $N^n$ is a subpresheaf of $\nu^n$. Let $X \in \Sm_k$ be connected with generic point $\eta$, and consider the following diagram
\begin{equation*}
\begin{tikzcd}
	N^n(X) \ar[dashed]{r}{\tilde f} \ar[hook]{d} & (\underline{I}^n/\underline{I}^{n+1})(X) \ar[hook]{d} \\
	\nu^n(\eta) \ar{r}{f}[swap]{\simeq} & (\underline{I}^n/\underline{I}^{n+1})(\eta)\rlap,
\end{tikzcd}
\end{equation*}
where $\underline{I}^n/\underline{I}^{n+1}$ is the quotient in the Nisnevich topology.
It follows from \cite[Theorem 8.3]{geisser-levine} that $\nu^n$ is strictly $\A^1$-invariant, and hence the left-hand vertical map is injective. Since strictly $\A^1$-invariant sheaves form an abelian subcategory of Nisnevich sheaves, $\underline{I}^n/\underline{I}^{n+1}$ is strictly $\A^1$-invariant and the right-hand vertical map is an injection.
By \cite{kato1982symmetric}, there is an isomorphism $f$ as displayed, which is determined by $f([a_1, \dotsc, a_n]) = \lAngle a_1, \dotsc, a_n \rAngle$.
This formula implies that $f$ restricts to a morphism $\tilde f$, and that the maps $\tilde f$ (for various $X$) assemble into a morphism of presheaves.
By \cite[Theorem 1.2]{morrow2015k}, $L_\zar N^n = \nu^n$ and hence we have constructed a map $L_\zar \tilde f\colon \nu^n \to \underline{I}^n/\underline{I}^{n+1}$.
Since $L_\zar \tilde f$ induces an isomorphism on fields, it is an isomorphism.

Let $R$ be an essentially smooth local $k$-algebra. For $n\geq 1$, let $J^n \subset I^n(R)$ denote the subgroup generated by $n$-fold Pfister forms.
Consider the exact sequence \[0 \to \underline{I}^{n+1}(R) \to \underline{I}^n(R) \to (\underline{I}^n/\underline{I}^{n+1})(R) \simeq \nu^n(R).\]
The map $J^n \subset \underline{I}^n(R) \to \nu^n(R)$ is surjective, since it sends $\lAngle a_1, \dotsc, a_n \rAngle$ to $[a_1, \dotsc, a_n]$.
It follows that $\nu^n(R) \simeq \underline{I}^n(R)/\underline{I}^{n+1}(R)$.
It also follows that $\underline{I}^n(R) = J^n + \underline{I}^{n+1}(R)$ and hence by induction that $\underline{I}^n(R) = J^n + \underline{I}^m(R)$ for any $m>n$.
By \cite[Theorem III.5.10 and preceeding paragraph]{HM} and the unramifiedness of $\underline{I}^m$, we have $\underline{I}^m(R) = 0$ for $m > \dim R$.
Thus $J^n = I^n(R) = \underline{I}^n(R)$, and so $I^n = \underline{I}^n$.
If $G \subset \underline{GW}(R)$ denotes the subgroup generated by $\langle a \rangle$ for $a\in R^\times$, then $G\to \Z$ is surjective and we similarly get $\underline{GW}(R) = G + I(R) = G$.
All claims follow.
\end{proof}

\begin{rem}
One may prove that for any local ring $R$, $\GW(R)$ is generated by elements of the form $\langle a \rangle$ for $a \in R^\times$.
It follows from Lemma \ref{lem:I-generators}(iii) that for any regular local $k$-algebra $R$ with fraction field $K$, the ring $\underline{GW}(R) \subset \underline{GW}(K) = \GW(K)$ coincides with the image of $\GW(R) \to \GW(K)$.
\end{rem}

\begin{prop}\label{prop:main-fr}
	Let $k$ be a field. Then the canonical maps
	\begin{gather*}
	L_\mot(\Z\times\FSYN_\infty^\fr)^+ \to L_\mot(\FSYN^{\fr,\gp})\\
	\Z\times (\Lhtp\FSYN^\fr_\infty)(k)^+ \to (\Lhtp\FSYN^\fr)(k)^\gp
	\end{gather*}
	are equivalences. The same holds for $\FSYN^\ornt$.
\end{prop}

\begin{proof}
	Let $f\colon L_\mot(\Z\times\FSYN_\infty^\fr) \to L_\mot(\FSYN^{\fr,\gp})$ be the canonical map. As in the proof of Corollary~\ref{cor:main}, $L_\mot(\FSYN^{\fr,\gp})$ is the group completion of $L_\mot(\FSYN^{\fr})$ in Nisnevich sheaves.
	For the first equivalence, it suffices by Corollary~\ref{cor:group completion} to show that $f$ induces an isomorphism on $\pi_0^\nis$. 
	The $\pi_0^\nis$ of the right-hand side is isomorphic to the unramified Grothendieck–Witt sheaf $\underline{\GW}$, and this isomorphism sends $\langle a\rangle$ to $\langle a\rangle$ \cite[Corollary 3.3.11]{EHKSY2}. If $X$ is an essentially smooth henselian local scheme, then $\underline{\GW}(X)$ is generated by $\langle a\rangle$ for $a\in\sO(X)^\times$ by Lemma~\ref{lem:GW-generators}, and since $\langle a\rangle$ is invertible in the left-hand side by Lemma~\ref{lem:GW-relation}, we deduce that $f$ is an effective epimorphism. Note that a surjective map of discrete monoids whose codomain is a group is injective if and only if its kernel is trivial. It therefore remains to show that the fiber of $f$ over $0$ is connected. By \cite[Lemma 6.1.3]{morel-connectivity}, it suffices to check this on finitely generated separable field extensions of $k$. Thus, it suffices to show that $\pi_0L_\mot(\Z\times\FSYN_\infty^\fr)(k)$ is a group for any field $k$; indeed, by Corollary~\ref{cor:group completion} this implies that $f$ is acyclic on $k$-points, and in particular connected. Since we have a surjection
	\[
	\pi_0\Lhtp(\Z\times\FSYN_\infty^\fr)(k)\to \pi_0L_\mot(\Z\times\FSYN_\infty^\fr)(k),
	\]
	it suffices to show that the left-hand side is a group, which also implies the second equivalence (by Corollary~\ref{cor:group completion} again). 
	
	Let $T\in\FSYN^\fr_d(k)$ be a framed finite syntomic $k$-scheme of degree $d$. Choosing an embedding of $T$ in $\A^m$ and lifting the framing of $T$ to a trivialization of its conormal sheaf, we obtain a $k$-point of the smooth $k$-scheme $\Hilb^\fr_d(\A^m)$. We now apply Proposition~\ref{prop:moving} with $X=\Hilb^\fr_d(\A^m)$ and $U\subset X$ the finite étale locus: we obtain framed finite étale $k$-schemes $A$ and $B$ with a framed cobordism $T\sqcup A\sim B$. Now $B$ is a sum of framed finite étale schemes of the form $(\Spec L,\langle a\rangle)$ with $L/k$ a finite separable field extension and $a\in L^\times$. We are therefore reduced to proving that $(\Spec L,\langle a\rangle)$ is invertible in $\pi_0\Lhtp(\Z\times\FSYN_\infty^\fr)(k)$. By Lemma~\ref{lem:GW-relation}, it is enough to prove that, for \emph{some} $a\in L^\times$, both $(\Spec L,\langle a\rangle)$ and $(\Spec L,\langle -a\rangle)$ are invertible. If $f(x)$ is a monic polynomial such that $L\simeq k[x]/(f(x))$, then $(\Spec L,\langle\pm f'(x)\rangle)\simeq \phi(\pm f)$. By \cite[Proposition B.1.4]{EHKSY1}, $\phi(\pm f)$ is framed cobordant to $\langle\pm 1\rangle[L:k]_\epsilon$, which is clearly invertible in $\pi_0(\Z\times\FSYN_\infty^\fr)(k)$. This completes the proof for $\FSYN^\fr$. The proof for $\FSYN^\ornt$ is exactly the same, using that the unit map $\Omega^\infty_\T\1\to \Omega^\infty_\T\MSL$ is a $\pi_0^\nis$-isomorphism (see~\cite[Example~16.35]{norms} and~\cite[Corollary~3]{muraMSL}).
\end{proof}

Except for the statement about positive characteristic, Theorem~\ref{thm:intro2} follows from Proposition~\ref{prop:main-fr}, \cite[Theorem~3.5.18]{EHKSY1}, and \cite[Corollary~3.4.4(i)]{EHKSY3}.

\section{Removing the plus construction in positive characteristic}
\label{sec:remove-plus}

The following proposition is an elaboration of \cite[Proposition 2.19]{Robalo}.

\begin{prop}
	\label{prop:robalo}
	Let $C$ be an $\sE_k$-monoidal $\infty$-category with $2\leq k\leq\infty$ and let $x\in C$.
	Let $C[x^{-1}]$ be the $\sE_k$-monoidal $\infty$-category obtained from $C$ by inverting $x$, and let $\tel_x(C)$ be the $C$-module colimit of the sequence
	\[
	\dotsb\to C \xrightarrow{\otimes x} C \to \dotsb.
	\]
	Consider the following assertions:
	\begin{enumerate}
		\item[\rm (1)] The cyclic permutation of $x^{\otimes 3}$ becomes homotopic to the identity in $\tel_x(C)$.
		\item[\rm (2)] For some $n\geq 2$, the cyclic permutation of $x^{\otimes n}$ becomes homotopic to the identity in $\tel_x(C)$.
		\item[\rm (3)] The canonical map $\tel_x(C)\to C[x^{-1}]$ is an equivalence.
	\end{enumerate}
	Then $(1)\Rightarrow (2)\Rightarrow (3)$. If $k\geq 3$, then $(3)\Rightarrow (1)$.
\end{prop}

\begin{proof}
	The implication $(1)\Rightarrow(2)$ is trivial. 
	Note that assertion $(3)$ is equivalent to the assertion that $x$ acts invertibly on the telescope $\tel_x(C)$.
	 Consider the commutative diagram
	\[
	\begin{tikzcd}[column sep=4.5em]
		C \ar{r}{ x^{\otimes (n-1)}} \ar{d}[swap]{ x} & C \ar{r}{ x^{\otimes (n-1)}} \ar{d}[swap]{ x} & \dotsb \\
		C \ar{r}{ x^{\otimes (n-1)}} & C \ar{r}{ x^{\otimes (n-1)}} & \dotsb
	\end{tikzcd}
	\]
	where each square commutes via the cyclic permutation of $x^{\otimes n}$. Let $\mathrm{Seq}$ be the $1$-skeleton of the nerve of the poset $\N$ and let $\sigma\colon \mathrm{Seq}{}^\triangleright \to\InftyCat$ be the cone given by the first row mapping to the colimit of the second row. Then the action of $x$ on $\tel_x(C)$ is the induced map $\colim(\sigma|\mathrm{Seq})\to \sigma(\infty)$.
	Under assumption (2), we obtain an equivalent cone if we replace each cyclic permutation in the above diagram by the identity, which is trivially a colimiting cone. This proves $(2)\Rightarrow (3)$.
	
	If $C$ is $\sE_3$-monoidal and $x\in C$ is invertible, then the cyclic permutation of $x^{\otimes 3}$ is homotopic to the identity. Indeed, $\pi_0\Aut(x^{\otimes 3})$ is an abelian group and the cyclic permutation of order $3$ becomes trivial in the abelianization of $\Sigma_3$. This proves $(3)\Rightarrow (1)$.
\end{proof}

\begin{ex}
	Let $R$ be a derived commutative ring. The cyclic permutation of $R^3$ is induced by a matrix in $SL_3(\Z)$ and hence is $\A^1$-homotopic to the identity. Applying Proposition \ref{prop:robalo} to the $\Einfty$-space $(\Lhtp\Vect)(R)$, we deduce the well-known fact that the canonical map 
	\[\underline\Z\times\Vect_\infty \to K\] 
	is an $\A^1$-equivalence on derived commutative rings, where $\underline\Z$ is the constant sheaf with value $\Z$ and $\Vect_\infty=\colim_n\Vect_n$. This explains why the plus construction is not needed in the equivalences (over a regular base)
	\[
	\Omega^\infty_\T\KGL\simeq L_\mot(\Z\times \Vect_\infty)\simeq L_\mot(\Z\times\Gr_\infty(\A^\infty)).
	\]
\end{ex}

\begin{lem}\label{lem:cyclic}
	Let $p$ be a prime number. Over $\F_p$, there exists a sequence of $C_p$-equivariant framed cobordisms between $p$ with nontrivial $C_p$-action and $\langle -1\rangle p_\epsilon$ with trivial $C_p$-action. In particular, for any $\F_p$-scheme $S$ and any $\alpha\in\FSYN^\fr(S)$, the action of $C_p$ on $p\alpha$ in $(\Lhtp\FSYN^\fr)(S)$ is trivial.
\end{lem}

\begin{proof}
	Let $f(t,x)=x(x-t)(x-2t)\dotsb(x-(p-1)t)$. 
	Consider the finite syntomic $\F_p[t]$-scheme
		\[
		X=\Spec \F_p[t,x]/(f(t,x)),
		\]
		with the framing $\tau$ induced by $-f(t,x)$.
		The group $C_p$ acts on $X$ over $\F_p[t]$ by $x\mapsto x+t$. As this action fixes $-f(t,x)$, we obtain an action of $C_p$ on $(X,\tau)$.
		The fiber $X_1$ consists of $p$ points cyclically permuted and trivially framed: indeed, the framing on the point $\Spec\F_p[x]/(x-i)$ is induced by $(-1)^{p}(p-1)!(x-i)=x-i$. On the other hand, the fiber $X_0$ is $\Spec\F_p[x]/(x^p)$ with trivial action and framing induced by $-x^p$. By \cite[Proposition B.1.4]{EHKSY1}, $X_0$ is framed cobordant to $\langle -1\rangle p_\epsilon$.
\end{proof}

\begin{rem}
	Let $S$ be an affine $\F_p$-scheme, $\sM$ an étale sheaf of $\Einfty$-spaces on $\Sm_S$, and $m\in\sM(S)$. Then the action of $C_p$ on $m^p$ in $(\Lhtp \sM)(S)$ is trivial, because $B_\et C_p$ is $\A^1$-contractible on affine $\F_p$-schemes by Lemma~\ref{lem:contr} below. This does not apply directly to $\FSYN^\fr$, which is not an étale sheaf. However, by \cite[\sectsign 4.2.36]{EHKSY1}, there is an $\Einfty$-map $\FET\to \FSYN^\fr$ where $\FET$ is the moduli stack of finite étale schemes, which is an étale sheaf. This gives an alternative proof of the last statement of Lemma~\ref{lem:cyclic}.
\end{rem}

\begin{lem} \label{lem:contr} 
	The presheaf $B_\et C_p$ is $\A^1$-contractible on affine $\F_p$-schemes.
\end{lem}

\begin{proof} 
	By the Artin–Schreier sequence \cite[Tag 0A3J]{stacks}, we have a fiber sequence of presheaves \[B_{\et}C_p \rightarrow B_{\et}\bG_a \rightarrow B_{\et}\bG_a,\] where $\bG_a$ is the additive group scheme. Since $\Lhtp B_{\et} \bG_a \simeq \ast$, it suffices to prove that this sequence remains a fiber sequence after applying $\Lhtp$. This follows from a general criterion for the geometric realization of a cartesian square to remain cartesian \cite[Lemma 5.5.6.17]{HA}, which applies because $\pi_0(B_{\et}\bG_a(X)) = H^1_\et(X, \sO) = *$ for $X$ an affine scheme \cite[Tags 03P2 and 01XB]{stacks}.
\end{proof}

The following proposition completes the proofs of all the theorems in the introduction.

\begin{prop}
	Let $k$ be a field of positive characteristic. Then the canonical maps
	\begin{gather*}
	L_\mot(\FSYN_\infty) \to L_\mot(\FSYN_\infty)^+\\
	(\Lhtp\FSYN_\infty)(k) \to (\Lhtp\FSYN_\infty)(k)^+
	\end{gather*}
	are equivalences. The same holds for $\FSYN^\fr_\infty$ and $\FSYN^\ornt_\infty$.
\end{prop}

\begin{proof}
	By Lemma~\ref{lem:cyclic} and Proposition~\ref{prop:robalo}, the left-hand sides are commutative monoids, and they are grouplike by Corollary~\ref{cor:main} and Proposition~\ref{prop:main-fr}. Hence they coincide with their plus constructions.
\end{proof}

\begin{rem}
	Let $C$ be a smooth curve over a field $k$ of characteristic zero. If $X$ is a finite flat $C$-scheme with an action of the cyclic group $C_n$, then the locus of points $c\in C$ such that $C_n$ acts trivially on the fiber $X_c$ is clopen. Indeed, that locus is the equalizer of a pair of sections of a finite flat $C$-subgroup scheme of $\Aut(X/C)$, which is necessarily étale.	
	In particular, unlike in positive characteristic, there cannot exist a $C_n$-equivariant cobordism between a finite syntomic $k$-scheme with nontrivial $C_n$-action and one with trivial $C_n$-action. The group homomorphism $C_n\to \pi_1(\Lhtp\FSYN_\infty)(k)$ could nevertheless be trivial for other reasons.
\end{rem}

%%%%%%%%%

%%%%%%%%%%%%%%%%%%%%%
% References
%%%%%%%%%%%%%%%%%%%%%

%\bibliographystyle{epigaref}
%\bibliography{Biblio}

\newcommand{\etalchar}[1]{$^{#1}$}
\providecommand{\bysame}{\leavevmode\hbox to3em{\hrulefill}\thinspace}
\providecommand{\MR}{\relax\ifhmode\unskip\space\fi MR }
\providecommand{\MRhref}[2]{%
  \href{http://www.ams.org/mathscinet-getitem?mr=#1}{#2}
}
\providecommand{\href}[2]{#2}

\end{document}